\documentclass{birkjour}
\usepackage{enumitem}
 \newtheorem*{unthm}{Theorem}
 \newtheorem{thm}{Theorem}[section]

 \newtheorem{prop}[thm]{Proposition}
 \theoremstyle{definition}
 
 \theoremstyle{remark}
 \newtheorem{rem}[thm]{Remark}
 \newtheorem*{ex}{Example}
 \numberwithin{equation}{section}

\begin{document}

%-------------------------------------------------------------------------
% editorial commands: to be inserted by the editorial office
%
%\firstpage{1} \volume{228} \Copyrightyear{2004} \DOI{003-0001}
%\seriesextraline{This is the Concrete Title of this Book\br H.E. R and S.T.C. W, Eds.}
%\firstpage{1}

\title[Cosmic convergence of nonexpansive maps]
 {Comments on the cosmic convergence\\ of nonexpansive maps}

%----------Author 1
\author[A.W. Guti\'{e}rrez]{Armando W. Guti\'{e}rrez}

\address{INRIA Saclay and CMAP Ecole Polytechnique\br 
CNRS, 91128, Palaiseau, France}
\email{armando.w.gutierrez@inria.fr}
\email{armando.gutierrez@polytechnique.edu}

%----------Author 2
\author[A. Karlsson]{Anders Karlsson}

\address[1]{Section de mathématiques, Université de Genève\br
2-4 Rue du Lièvre, Case Postale 64, 1211 Genève 4, Suisse}
\email{anders.karlsson@unige.ch}

\address[2]{Matematiska institutionen, Uppsala universitet\br
Box 256, 751 05 Uppsala, Sweden}
\email{anders.karlsson@math.uu.se}

%\thanks{}

\subjclass{Primary 47H09; Secondary 47A15}

\keywords{Metric fixed point theory, nonexpansive maps, invariant subspace problem, metric functionals}

\date{August 10, 2021}

\begin{abstract}
This note discusses some aspects of the asymptotic behaviour of nonexpansive
maps. Using metric functionals, we make a connection to the invariant
subspace problem and prove a new result for nonexpansive maps of $\ell^{1}$.
We also point out some inaccurate assertions appearing in the literature on 
this topic.
\end{abstract}

%%% ----------------------------------------------------------------------
\maketitle
%%% ----------------------------------------------------------------------
%\tableofcontents
\section{Introduction}

This brief note was inspired by recent papers by Bauschke, Douglas,
and Moursi \cite{BDM16}, and Ryu \cite{Ry18}. Let $(X,d)$ be a metric
space and let $T$ be a nonexpansive map of $X$ into itself, that is
\[
d(Tx,Ty)\leq d(x,y)
\]
for all $x,y\in X$. The question concerns the asymptotic behaviour
of the iterates $T^{n}x$ as $n\rightarrow\infty$. The first important 
case of this is well-known: the contraction mapping principle.
In this paper we are mostly interested in the complementary case; 
when no fixed point exists. The case when $X$ is a Banach space has 
been especially studied, see \cite{P71,R73,BBR78,KoN81,R81,R82,PlR83,MR03,BDM16,Ry18} 
and references therein. It should be pointed out that a significant 
special case was considered in the 1930s: the mean ergodic theorem 
of von Neumann and Carleman, especially in the generality of 
F. Riesz \cite[Theorem~1]{R38}. Here $X$ is a Hilbert space and 
\[
Tx=Ux+v
\]
where $U$ is a linear operator of norm at most one and $v\in X$.
Then the iterates $T^{n}0$ converge in the following sense
\[
\lim_{n\to\infty}\frac{1}{n}T^{n}0 = \lim_{n\to\infty}\frac{1}{n}\sum_{k=0}^{n-1}U^{k}v = \pi_{U}(v)
\]
where $\pi_{U}$ is the projection onto the subspace of $U$-invariant
vectors.

It is stated in \cite[Fact~1.1]{BDM16} and also in \cite{Ry18} that
for fixed-point free nonexpansive maps of Hilbert spaces the orbit
must diverge in the sense that $\left\Vert T^{n}x\right\Vert \rightarrow\infty$
as $n\rightarrow\infty$. This may however fail in infinite-dimensional 
Hilbert spaces as an isometry discovered by Edelstein in the 1960s 
demonstrates, recalled in \cite[p.~1453]{K01} and the nice exposition \cite[p.~6]{V07}. 
That is, in spaces which are not locally compact it may happen that 
the orbit $(T^{n}x)_{n\geq 0}$ neither stays bounded nor leaves every bounded set. 
On the other hand, this phenomenon cannot happen in spaces where
closed bounded sets are compact as it was shown by Calka in the 1980s 
(a simple proof is given in \cite[Lemma 2.6]{G18}).

The quantity 
\begin{equation}
	\tau=\lim_{n\rightarrow\infty}\frac{1}{n}\left\Vert T^{n}x\right\Vert \label{eq:tau}
\end{equation}
always exists by a well-known subadditive argument and it is independent
of $x$. Convergence in direction, \emph{cosmic convergence }in the
terminology of \cite{BDM16,Ry18}, in the case $\tau>0$ for general
Banach spaces was treated in \cite{R73,KoN81} among other papers.
A very general result in this direction is found in \cite{GK20} which
considers random products of nonexpansive maps of any metric space.
In the case $\tau=0$ one may wonder, as the authors in \cite{BDM16}
do, if cosmic convergence still takes place, that is, whether
\[
\frac{T^{n}x}{\left\Vert T^{n}x\right\Vert }
\]
converges as $n\rightarrow\infty$. Strong and weak limits of this
expression are called strong and weak cosmic limits respectively.
This question is also raised in \cite[Problem~4.6]{K02}. The paper
\cite{BDM16} proves several interesting cases when this is true,
and suggests a conjecture that it might always be true for finite
dimensional Hilbert spaces. A counterexample was given in \cite[Section~3]{Ry18}.
In 2005, Enrico Le Donne showed the second author another such 
an example (unpublished). In other normed spaces, a counterexample 
was considered by Kohlberg and Neyman in \cite[p.~272]{KoN81}, 
also discussed in \cite[p.~1936]{G18}.

It still makes sense to wonder about what the limit set can be when
it is not just one point. The theorem \cite[Theorem~11]{K05} shows
that it must be contained in certain kinds of sectors. This can be
compared with Corollary 5.3 to Theorem 5.1 in \cite{Ry18}. These
are some of the most general results presently available on this type
of convergence known to us. When $X$ is hyperbolic there are many
results starting from Wolff-Denjoy showing that cosmic convergence
holds, for example \cite{K01}, see also the discussion in \cite{BDM16}.
Let us point out the similar question and conjecture in the case $X$
is a convex set equipped with Hilbert's metric, called the Karlsson-Nussbaum
conjecture, see \cite{LLNW18} for one of the most recent significant
contributions.

In Section~\ref{sec:3} we will prove the following connection of cosmic convergence to the problem of
the existence of non-trivial closed invariant subspaces:
\begin{unthm}
	\label{thm:invariant}
	Suppose that $T$ is a nonexpansive map of
	a real Hilbert space into itself of the form $Tx=Ux+v$ for some vector $v$ 
	and linear operator $U$ of norm at most one. If $0\notin\overline{Im(I-T)},$
	then there is a non-zero continuous linear functional $f$ such that
	\[
	f(Tx)\geq f(x)
	\]
	for all $x$. This inequality implies that $\ker(f)$ is a non-trivial
	closed invariant subspace for $U$.  
\end{unthm}

In Section~\ref{sec:4} we will prove:
\begin{unthm}
	Let $T$ be a nonexpansive map of $\ell^{1}(\mathbb{Z})$ into itself. 
	Then there is a non-trivial metric functional $h$ so that 
	\[
	h(Tx)\leq h(x)
	\]
	for every $x\in\ell^{1}(\mathbb{Z})$, and the orbit $(T^{n}0)_{n\geq 0}$
	is contained in a half-space. More precisely, there is a non-trivial
	continuous linear functional $f$ so that 
	\[
	f(T^{n}0)\geq 0
	\]
	for all $n\geq 0$. 
\end{unthm}

Metric functionals, which constitute a main tool in this note will
be recalled below. We provide some further statements and examples,
such as Proposition \ref{prop:firmly} about firmly nonexpansive maps.
We end by giving a simple example of a nonexpansive map whose unbounded
orbits converge to a metric functional but without having any
weak cosmic limit point. This shows once more that metric functionals
are useful when linear notions fail to describe a phenomenon. 

\section{Metric functionals}

Let $(X,d)$ be any metric space, for example a Banach space. Fix
$x_{0}\in X$, in the case of a Banach space we take $x_{0}=0$. 
We consider the following map 
\[
\Phi:X\rightarrow\mathbb{R}^{X}
\]
via
\[
x\mapsto h_{x}(\cdot):=d(\cdot,x)-d(x_{0},x).
\]
The map $\Phi$ is injective (if $h_x = h_y$, then in particular
$h_{x}(x)=h_{y}(x)$ and $h_{x}(y)=h_{y}(y)$ which together imply 
that $d(x,y)=0$). As the notation indicates, we endow the target 
space $\mathbb{R}^{X}$ with the topology of pointwise convergence. 
The map $\Phi$ is clearly continuous. 
The closure $\overline{\Phi(X)}$ is compact and consists of a subset of 
nonexpansive maps $h: X \to \mathbb{R}$ with $h(x_{0})=0$.
Every element of $\overline{\Phi(X)}$ is called a \emph{metric functional}
on $X$. See \cite{GV12,GK20,Gu19,K19} for discussions on metric functionals
and the related notion of horofunctions. 

Below we will use two cases:
\begin{prop}
	\label{prop:hilbert}\cite{K19,Gu19} Let $H$ be a real Hilbert space
	with scalar product $(\cdot,\cdot)$. Every metric functional on $H$
	has precisely one of the following forms:
	\begin{enumerate}[label=\arabic*.]
	\item $h(x)=\left\Vert x\right\Vert$;
	\item $h(x)=\sqrt{\left\Vert x\right\Vert ^{2}-2(x,rv)+r^{2}}-r$, with $0< r <\infty$ 
			and $v\in H$, $\left\Vert v\right\Vert <1$;
	\item $h(x)=\left\Vert x-rv\right\Vert -r$, with $0< r <\infty$ 
			and $v\in H$, $\left\Vert v\right\Vert =1$;
	\item $h(x)=-(x,v)$, with $v\in H$, $\left\Vert v\right\Vert \leq 1$.
	\end{enumerate}
\end{prop}

Note that only the metric functionals of the form \emph{4.} are unbounded from below (except
$h\equiv 0$). 

The first author determined all the metric functionals on $\ell^{1}(\mathbb{Z})$:
\begin{prop}
	\label{prop:l1}\cite{Gu19} The following functions are precisely
	the metric functionals on $\ell^{1}(\mathbb{Z})$:
	\[
	h(x)=\sum_{s\in I}\epsilon_{s}x_{s}+\sum_{s\in\mathbb{Z}\setminus I}\left|x_{s}-z_{s}\right|-\left|z_{s}\right|
	\]
	where $I\subseteq\mathbb{Z}$, $\epsilon_{s}\in\left\{ -1,+1\right\} $
	for all $s\in I$ and $z_{s}$ are arbitrary real numbers for all
	$s\in\mathbb{Z}\setminus I$. 
\end{prop}

Another fact we will use below (see \cite[Lemma~3.1]{GK20}) is that
for any metric functional $h$ on a Banach space there is always a
continuous linear functional $f$ of norm at most 1 such that $f\leq h.$
As a first illustration of this tool let us point out the following.
Browder and Bruck introduced \emph{firmly nonexpansive maps }of Banach
spaces. Those are maps $T:V\rightarrow V$ of a Banach space $V$(or
a convex subsets thereof) satisfying 
\[
\left\Vert Tx-Ty\right\Vert \leq\left\Vert (1-t)(Tx-Ty)+t(x-y)\right\Vert 
\]
for every $x,y\in V$ and $t\geq0$. Reich and Shafrir proved in \cite{RS87}
that for such maps it holds that
\[
\lim_{n\rightarrow\infty}\left\Vert T^{n+1}x-T^{n}x\right\Vert =\tau,
\]
where $\tau$ is defined above in (\ref{eq:tau}). This is
a property used in the proof of \cite[Theorem~5.1]{Ry18}. The following
reproves the main theorem in \cite[Section~3]{GV12} in a very special case,
but with the additional information that our metric functional $h$ is a limit point of
the orbit $(T^{n}0)_{n\geq 0}$ in the compact space $\overline{\Phi(V)}$. 
\begin{prop}
	\label{prop:firmly}Let $T$ be a firmly nonexpansive map of a Banach
	space $V$ into itself with $\tau=0$. Then there is a metric functional $h$
	on $V$ which is a limit point of $(T^{n}0)_{n\geq 0}$ in the compact 
	space $\overline{\Phi(V)}$ and such that 
	\[
	h(Tx)\leq h(x)
	\]
	for all $x\in V.$
\end{prop}

\begin{proof}
	Take a subsequence $n_{i}$ so that $h_{T^{n_{i}}0}\rightarrow h$
	(in the case that $\overline{\Phi(V)}$ is not sequentially compact we can use
	a similar argument as in \cite[p.~1907]{GK20}). Note that for every $x\in V$,
	\begin{align*}
		h(Tx) &=\lim_{i\rightarrow\infty}\left\Vert Tx-T^{n_{i}}0\right\Vert -
					\left\Vert T^{n_{i}}0\right\Vert \\
			  &\leq\liminf_{i\rightarrow\infty}\left\Vert x-T^{n_{i}-1}0\right\Vert -
			  		\left\Vert T^{n_{i}}0\right\Vert \\
			  &\leq\liminf_{i\rightarrow\infty}\left\Vert x-T^{n_{i}}0\right\Vert +
			  		\left\Vert T^{n_{i}}0-T^{n_{i}-1}0\right\Vert -
			  		\left\Vert T^{n_{i}}0\right\Vert \\
			  &=h(x)+\tau.
	\end{align*}
	If $\tau=0$, then the proposition is proved. 
\end{proof}

Note that we do not need $T$ to be defined on the whole vector space,
just as long as we can iterate the map. Firmly nonexpansive maps are
of interest in the study of certain non-linear operators as pioneered
by Browder and they often arise in optimization problems; see, for example,
\cite{R77,BR77,BMR04}.

Ryu constructed an interesting firmly nonexpansive map $T$ 
of $\ell^{2}(\mathbb{N})$ into itself \cite[p.~11]{Ry18} 
such that $T^{n}0/\left\Vert T^{n}0\right\Vert$ converges weakly, but
not strongly, to $0$. We note that the iterates $T^{n}0$ converge to the metric
functional $h=0$. Furthermore we note that Ryu's map also provides
a nonexpansive map of $\ell^{1}(\mathbb{N})$ into itself,
and that in this case the iterates $T^{n}0$ converge to a non-trivial
metric functional (which in fact is linear). Indeed, the
result in \cite[Lemma~4.1]{Ry18} implies that for fixed $i$, 
the $i$th coefficient of the $k$th iterate looks 
like $(T^{k}0)_{i}=\log k+O(1)$ as $k\rightarrow\infty$.
Now if we consider $x\in\ell^{1}$ for which only a finite number
of coefficients are non-zero, say $F$ is the support of $x$, then
\[
\left\Vert x-T^{k}0\right\Vert -\left\Vert T^{k}0\right\Vert =\sum_{i\in F}\left|x_{i}-(T^{k}0)_{i}\right|-\left|(T^{k}0)_{i}\right|
\]
which converges to $-\sum_{i\in F}x_{i}$ as $k\rightarrow\infty$.
Since such finitely supported points are dense in $\ell^{1}$ we have
that the iterates $T^{k}0$ converge to the metric functional, indeed linear,
$h(x)=-\sum_{i\ge0}x_{i}$. 

\section{Invariant subspaces}\label{sec:3}

One of the oldest and best-known open problems in operator theory
is the invariant subspace problem. It asks whether every bounded linear
operator of a complex Hilbert space $H$ of dimension at least two
must have a non-trivial invariant closed linear subspace. For general
Banach spaces the first counterexample was found by Enflo and other
examples by Read \cite{Re85}. We will here make a connection to the
topic of the present paper.

First we point out that in Remark 3.1 of \cite{KoN81} it is stated
that one can choose a linear functional $f$ of norm $1$ such that
$f(T^{n}x-x)\geq n\tau.$ But it seems to us that this can only be
guaranteed when $\tau>0$, otherwise the weak limit of linear functionals
in the proof may be $f\equiv0.$ Compare also with \cite[p.~355]{GV12}.

Given any bounded linear operator $U:H\rightarrow H$, by rescaling
we can assume that its norm is at most one since this does
not alter the invariant subspaces. We will consider associated
affine nonexpansive maps, that is, $Tx=Ux+v$ for some vector $v\in H$.
We will be interested in statements of the form that there exists
a linear functional $f$ such that
\[
f(Tx)\geq f(x)
\]
for all $x\in H.$
\begin{thm}
	\label{prop:invariant}Suppose that $T$ is a nonexpansive map of
	a real Hilbert space into itself of the form $Tx=Ux+v$ for some vector $v$ 
	and linear operator $U$ of norm at most one. If $0\notin\overline{Im(I-T)},$
	then there is a non-zero continuous linear functional $f$ such that
	\[
	f(Tx)\geq f(x)
	\]
	for all $x$. This inequality implies that $\ker(f)$ is a non-trivial
	closed invariant subspace for $U$. 
\end{thm}

\begin{proof}
	It is known since Pazy's work that the condition imposed implies that
	the corresponding escape rate $\tau$ defined above in (\ref{eq:tau})
	is strictly positive. This also follows from the main theorem in Gaubert-Vigeral
	\cite{GV12} that moreover asserts that there is a metric functional
	such that 
	\[
	h(Tx)\leq h(x)-\tau
	\]
	for every $x$ in the Hilbert space. Since $\tau>0$, it follows
	from the above inequality iterated that $h(T^{n}x)\rightarrow-\infty$
	as $n\rightarrow\infty$. This shows, in view of the identification
	of the metric functionals in a Hilbert space recalled in Proposition
	\ref{prop:hilbert}, that $h$ must be of the type $h(x)=-(x,q)$ for
	some non-zero vector $q$ which must have norm $1$. Hence $f(x):=-h(x)$
	is the linear functional as required.
	
	Notice now that this means that $-(Ux+v,q)\leq-(x,q)$ for every $x$.
	In particular, for any $x\in\ker f$, we have $(Ux,q)+(v,q)\geq 0.$
	Applying the same inequality with $tx$ with scalars $t\in\mathfrak{\mathbb{R}}$
	instead of $x,$ using linearity of the scalar product, it follows
	that the only possibility is that $(Ux,q)=0,$ in other words $Ux\in\ker f$.
	Hence the kernel of $f$, which clearly is a non-trivial and closed
	linear subspace is invariant under $U$. 
\end{proof}

There is a hope, related to cosmic convergence, 
that even when $0\in\overline{Im(I-T)}\setminus Im(I-T),$
there is sometimes such a linear functional $f.$ (The case $0\in Im(I-T)$
means that $v=x-Ux$ for a certain $x$). It is stated 
in \cite[Theorem~5.1]{Ry18} that for a non-zero cosmic weak limit point $q$, 
it holds that $(Tx-x,q)\geq0$ for all $x$ in the Hilbert space, which
is precisely what is needed in view of Theorem \ref{prop:invariant}.
Note however that in the proof of \cite[Theorem~5.1]{Ry18} it is assumed 
that $\left\Vert T^{k_{j}+1}x-T^{k_{j}}x\right\Vert \rightarrow0$,
but for an isometry this norm is constant and positive unless $x$
is a fixed point. As far as we can see the proof of \cite[Theorem~5.1]{Ry18}
needs this assumption.

\begin{rem}
	The condition in Theorem \ref{prop:invariant} means that $1$ is in the 
	compression spectrum of the linear operator $U$. 
	By the definition of this subset of the spectrum, the set $\overline{Im(I-U)}$
	is a non-trivial closed subspace which is clearly invariant under $U$. 
	In contrast, Theorem \ref{prop:invariant} provides a co-dimension $1$ 
	closed invariant subspace for $U$.
\end{rem}

\section{Nonexpansive maps of $\ell^{1}$ }\label{sec:4}

In view of Theorem \ref{prop:invariant} above, it is interesting
to recall that for any nonexpansive map $T$ of a Banach space, Gaubert
and Vigeral in their paper \cite{GV12} (strenghtening \cite{K01}
in the case of star-shaped hemi-metrics) provide a \emph{metric} functional
$h$ so that 
\[
h(Tx)\leq h(x)
\]
for every $x$. For a Hilbert space, in case $\tau=0$ there seems
to be no way in general of guaranteeing that $h$ is not the function
identically $0.$ On the other hand, Gutiérrez \cite{Gu19} identified
explicitly all the metric functionals of the Banach spaces $\ell^{1}$
recalled in Proposition \ref{prop:l1}, and none is identically $0$.
So this gives finer information than \cite{KoN81} in this case. It
is pointed out in \cite[Lemma~3.1]{GK20} that there is always a continuous
linear functional $f$ of norm at most 1 such that $f\leq h.$ Nevertheless,
the discrepancy between linear and metric functionals has as consequence
that the proof of Theorem \ref{prop:invariant} does not lead to an
affirmative solution to the invariant subspace problem in the case
of $\ell^{1}$. Indeed, there is a celebrated counterexample to the
invariant subspace problem for $\ell^{1}$, constructed by Read \cite{Re85}. 

The following can be said about nonexpansive maps of $\ell^{1}$:
\begin{thm}
	Let $T$ be a nonexpansive map of $\ell^{1}(\mathbb{Z})$ into itself. 
	Then there is a non-trivial metric functional $h$ so that 
	\[
	h(Tx)\leq h(x)
	\]
	for every $x\in\ell^{1}(\mathbb{Z})$, and the orbit $(T^{n}0)_{n\geq 0}$
	is contained in a half-space. More precisely, there is a non-trivial
	continuous linear functional $f$ so that 
	\[
	f(T^{n}0)\geq 0
	\]
	for all $n\geq 0$. 
\end{thm}

\begin{proof}
	The first statement is the Gaubert-Vigeral Theorem \cite[Theorem~1]{GV12}
	with the addition of the explicit determination of such metric functionals
	in \cite{Gu19}, see Proposition \ref{prop:l1} above. Each of these
	metric functionals takes on negative values somewhere, except for
	$h(x)=\left\Vert x\right\Vert $. In this latter exceptional case,
	note that $h(Tx)\leq h(x)$ means that $\left\Vert Tx\right\Vert \leq\left\Vert x\right\Vert $
	which applied to $x=0$ gives $T(0)=0.$ In this case any linear functional
	$f$ would satisfy our claim since $f(T^{n}0)=f(0)=0$. If $h$ does take negative
	values then take the linear functional $g$ obtained in \cite[Lemma~3.1]{GK20}
	with $g\leq h$. This linear functional must be non-trivial since
	it is forced to take on strictly negative values at some points. Iterating
	the main inequality we have 
	\[
	h(T^{n}0)\leq h(T^{n-1}0)\leq...\leq h(0)=0.
	\]
	Let finally $f(x)=-g(x)$, of course also not the identically 0 linear
	functional, and with the property $f(T^{n}0)\geq-h(T^{n}0)\geq0.$ 
	
\end{proof}

Let us remark that Edelstein's example alluded to already above extends
to $\ell^{1}(\mathbb{Z})$. This isometry has unbounded orbits but
they nevertheless return infinitely often to a fixed bounded set.
Let us finish by another example:
\begin{ex}
	\label{exa:cosmic}Consider the nonexpansive map $T:\ell^{1}(\mathbb{N})\rightarrow\ell^{1}(\mathbb{N})$
	defined by $T(x_{1},x_{2},...)=(1,x_{1},x_{2},...)$. This map clearly
	has no fixed points in $\ell^{1}(\mathbb{N})$. Indeed the orbits
	tend to infinity, more precisely, $T^{n}0=(1,1,1,...,1,0,0,...)$
	and $h_{T^{n}0}$ converges to
	\[
	h(x)=\sum_{s=1}^{\infty}\left|x_{s}-1\right|-1
	\]
	as $n\rightarrow\infty$. This $h$ is the metric functional that
	is obtained in \cite{K01,K19} (also the one in \cite{GV12} as can
	be verified), and clearly $h(T^{n}0)\rightarrow-\infty$. On the other
	hand, there are no weak cosmic limit points. 
\end{ex}

% ------------------------------------------------------------------------
\subsection*{Acknowledgment}
The first author acknowledges financial support from the 
Vilho, Yrj\"{o} and Kalle V\"{a}is\"{a}l\"{a} Foundation 
of the Finnish Academy of Science and Letters. 
The second author acknowledges partial financial support 
from the Swiss NSF grant 200020\_15958.

% ------------------------------------------------------------------------
\end{document}